\documentclass[12pt]{amsart}

\usepackage[top=4cm,bottom=4cm,inner=3cm,outer=3cm]{geometry}
\usepackage{tikz-cd}

\newcommand\mapsfrom{\mathrel{\reflectbox{\ensuremath{\mapsto}}}}
\newcommand{\nc}{\newcommand}
\nc{\G}{{\Gamma}} \nc{\BC}{{\mathbb C}} \nc{\BQ}{{\mathbb Q}}
\nc{\BR}{{\mathbb R}} \nc{\BZ}{{\mathbb Z}} \nc{\BP}{{\mathbb P}}
\nc{\BN}{{\mathbb N}} \nc{\BM}{{\mathbb M}}
\nc{\fH}{{\mathbb H}}

\nc{\U}{{\mathcal U}}
\nc{\PS}{{\mbox{PSL}_2(\BZ)}} \nc{\SL}{{\mbox{SL}_2(\BZ)}}
\nc{\SR}{{\mbox{SL}_2(\BR)}} \nc{\PR}{{\mbox{PSL}_2(\BR)}}
\nc{\GL}{{\mbox{GL}}} \nc{\PQ}{{\mbox{PGL}_2^+(\BQ)}}
\nc{\GR}{{\mbox{GL}_2^+(\BR)}} \nc{\PG}{{\mbox{PGL}_2^+(\BR)}}
\nc{\GC}{{\mbox{GL}_2(\BC)}}
\nc{\f}{{\mathcal{F}(\fH)}}
\nc{\Cc}{\widehat{\BC}}
\nc{\e}{{E_{\rho}(\G)}}
\nc{\g}{{\gamma}}
\nc{\vm}{{V_{\rho}(\G)}}
\nc{\oo}{{\mathcal O}}
\nc{\M}{{\mbox{M}}}
\nc{\om}{{\omega}}
\nc{\Om}{{\Omega}}
\nc{\TX}{{\widetilde{X}}}
\nc{\ol}{\overline}
\nc{\cl}{{\mathcal L}}
\nc{\ce}{{\mathcal E}}
\nc{\la}{{\lambda}}
\nc{\La}{{\Lambda}}
\nc{\cz}{{\mathcal Z}}

\newtheorem{numbered}{}[section]
\newtheorem{thm}[numbered]{Theorem}

\newtheorem{remark}[numbered]{Remark}
\newtheorem{prop}[numbered]{Proposition}

\numberwithin{equation}{section}

\newcommand{\thmref}[1]{Theorem~\ref{#1}}
\newcommand{\propref}[1]{Proposition~\ref{#1}}
\newcommand{\secref}[1]{\S\ref{#1}}

\begin{document}

\title{Elliptic Zeta functions and equivariant functions}
\author[]{Abdellah Sebbar} \author[]{Isra Al-Shbeil} 
\address{Department of Mathematics and Statistics, University of Ottawa, Ottawa Ontario K1N 6N5 Canada}
\email{asebbar@uottawa.ca}
\email{ialsh010@uottawa.ca}


\begin{abstract}
	In this paper we establish a close connection between three notions attached to a modular subgroup. Namely the set of weight two meromorphic modular forms, the set of equivariant functions on the upper half-plane commuting with the action of the modular subgroup and the set of elliptic zeta functions generalizing the Weierstrass zeta functions. In particular, we show that the equivariant functions can be parameterized by modular objects as well as by elliptic objects.
\end{abstract}
\maketitle
\section{Introduction}
For a finite index subgroup $\Gamma$ of $\SL$, an equivariant function is a meromorphic function on the upper half-plane $\fH$ which commutes with the action of $\Gamma$ on $\fH$. Namely,
\[
f(\gamma \tau)=\gamma f(\tau)\;,\ \gamma\in\Gamma\;,\ \tau\in\fH,
\]
where $\gamma$ acts by linear fractional transformations on both sides.
These were extensively studied in connection with modular forms in \cite{sb1,sb2,ss1} and have important applications to modular forms and vector-valued modular forms
\cite{ss3,ss2}. In this paper, we study the equivariant functions from an elliptic point of view. In particular, we will see  that they arise also from elliptic objects. To this end we establish
correspondences between three distinct notions. The first of which is the set of equivariant functions for $\Gamma$. The second is the space of weight 2 meromorphic modular forms $M_2(\Gamma)$.
The third set under consideration consists of a generalization of the Weierstrass $\zeta-$function which satisfies $\zeta'(z)=-\wp(z)$ where $\wp$ is the Weierstrass $\wp-$function
attached to a rank two lattice of $\BC$. In fact, $\zeta$ can be viewed as map
\[
\zeta:\,\{\mbox{set of lattices in } \BC\}\times \BC\mapsto \BC\cup\{\infty\}.
\]
For a fixed lattice $\omega_1\BZ+\omega_2\BZ$ with $\Im(\omega_2/\omega_1)>0$, the
map $\zeta(\omega_1\BZ+\omega_2\BZ,\;\cdot\;)$ is quasi-periodic in the sense that
\[
\zeta(\omega_1\BZ+\omega_2\BZ,z+\omega)=\zeta(\omega_1\BZ+\omega_2\BZ,z)+H(\omega),
\ z\in\BC,\ \omega\in \omega_1\BZ+\omega_2\BZ.
\]
Here $H(\omega$) does not depend on $z$ and is referred to as the quasi-period map. It is also $\BZ-$linear and thus it is completely determined by the quasi-periods
$H(\omega_1)$ and $H(\omega_2)$. Moreover, $\zeta$ is homogeneous in the sense that
\[
\zeta(\alpha(\omega_1\BZ+\omega_2\BZ),\alpha z)=\alpha^{-1}\zeta(\omega_1\BZ+\omega_2\BZ,z), \ \alpha\in\BC^*
\]
and so is the quasi-period map $H(\omega)$.

In the particular case where the lattice is of the form $\BZ+\tau \BZ$, $\tau\in\fH$,
 $H(1)$ and $H(\tau)$ are meromorphic as functions of $\tau$.
It turns out that the quotient $H(\tau)/H(1)$
 is an equivariant function on $\fH$ thanks to the linearity and the homogeneity of the quasi-period map $H$ \cite{brady}.

To generalize the Weierstrass $\zeta-$function, Brady, in {\em loc. cit.}
gave the definition of zeta-type functions that behave like $\zeta$ in terms of quasi-periodicity,  homogeneity, meromorphic behavior of the quasi-periods $H(1)$ and $H(\tau)$ and additional conditions.

In our case, we adapt and simplify these maps which we call elliptic zeta functions.
The quasi-periods $H(1)$ and $H(\tau)$ turn out to hold important information and they are used to construct equivariant functions as well as elements of $M_2(\SL)$.

If $\Gamma$ is a finite index subgroup of $\SL$, we generalize the above construction by
defining the notion of $\Gamma-$elliptic zeta functions. Here the lattices are replaced by appropriate classes involving $\Gamma$ which essentially can be identified with pairs of lattices $(L,L')$, $L'$ being a sub-lattice of finite index of $L$. The group $\SL$ acts by automorphisms on $L$ by change of basis and $\Gamma$ becomes the subgroup of $\SL$ that leaves $L'$ invariant. We then establish a triangular correspondence between the set of $\Gamma-$elliptic zeta functions, $M_2(\Gamma)$
and the set of $\Gamma-$equivariant functions summarized in the following commutative diagram in which every arrow is surjective.

\[
\begin{tikzcd}[row sep=3em, column sep=0.01em]
& \mbox{Weight 2 modular forms} \arrow{dr}{} \\
\mbox{Elliptic zeta functions} \arrow{ur}{} \arrow{rr}{} && \mbox{Equivariant functions}
\end{tikzcd}
\]

This paper is organized as follows:

In \secref{section2} we review the basic notions of periodic and quasi-periodic functions in the context of the Weierstrass $\wp$ and $\zeta$ functions. In \secref{section3}, inspired by M. Brady \cite{brady}, we introduce the notion of elliptic zeta functions and
study their structure. In \secref{section4}, we establish the connection between weight 2 modular forms and the elliptic zeta functions. In section \secref{section5},
we review the notion of equivariant functions and establish a correspondence with the
weight two modular forms. In \secref{section6} we generalize the constructions of the previous sections to any finite index subgroup of $\SL$. Finally, in \secref{section7} we provide some interesting examples related to the powers of the Weierstrass $\wp-$function.

\section{Quasi-periodic functions}\label{section2}
The main reference in this section is \cite{lang}.
Let $\La\subset \BC$ be a lattice in $\BC$, that is $\La=\omega_1\,\BZ \,+\,\omega_2\,\BZ$ with
$\displaystyle \Im (\omega_2/\omega_1)>0$. Such lattice can be expressed with a different basis $(\omega'_1,\omega'_2)$ if  $\omega'_1=a\,\omega_1+b\,\omega_2$ and $\omega'_2=c\,\omega_1+d\,\omega_2$ with
$\displaystyle \gamma=\begin{bmatrix}
a&b\\c&d
\end{bmatrix}\in\SL$, that is $(\omega'_1,\omega'_2)=(\omega_1,\omega_2)\gamma^t$ where $\gamma^t$ denotes the transpose of the matrix $\gamma$.
The Weierstrass $\wp-$function is the elliptic function with respect to $\La$  given by:
\[
\wp(\La,z)\,=\,\frac{1}{z^2}\,+\, \sum_{\substack{\omega\in \La\\ \omega\neq 0}}\,\left(\frac{1}{(z-w)^2}\,-\,\frac{1}{\omega^2}\right).
\]
It is absolutely and uniformly convergent on compact subsets of $\BC\setminus \La$ and defines a meromorphic function on $\BC$ with poles of order 2 at the points of $\La$ and no other poles.

The Weierstrass $\zeta-$function is defined by the series
\begin{equation}\label{zeta}
\zeta(\La,z)\,=\,\frac{1}{z}\,+\, \sum_{\substack{\omega\in \La\\ \omega\neq 0}}\,\left(\frac{1}{z-w}\,+\frac{1}{\omega}\,+\,\frac{z}{\omega^2}\right).
\end{equation}
It is absolutely and uniformly convergent on compact subsets of $\BC\setminus \La$. Moreover, it defines a meromorphic function on $\BC$ with simple poles at the points of $\La$ and no other poles. Differentiating the above series we get for all $z\in\BC$:
\[
\frac{d}{dz}\zeta(\La,z)\,=\,-\wp(\La,z)\;.
\]
Since $\wp$ is periodic relative to $\La$, $\zeta$ is quasi-periodic in the sense that for all $\omega\in\La$ and for all $z\in\BC$, we have
\begin{equation}\label{quasi}
\zeta(\La,z+\omega)\,=\,\zeta(\La,z)\,+\,\eta_{\La}(\omega),
\end{equation}
where $\eta_{\La}(\omega)$ is independent of $z$. We call $\eta_{\La}:\La\longrightarrow\BC$
the quasi-period map associated to $\zeta$. It is clear that $\eta_{\La}$ is $\BZ-$linear and thus it is completely determined by the values of
 $\eta_{\La}(\omega_1)$ and $\eta_{\La}(\omega_2)$.
 Also, since $\zeta$ is an odd function, it follows that if $\omega\in\La$ and
 $\omega\notin 2\La$, then $\eta_{\La}(\omega)$ is given by
 \begin{equation}\label{expr1}
 \eta_{\La}(\omega)\,=\,2\zeta(\Lambda,\frac{\omega}{2})\;.
 \end{equation}
 The periods and the quasi-periods are related by the Legendre relation:
 \begin{equation}\label{legendre}
 \omega_1\eta_{\La}(\omega_2)-\omega_2\eta_{\La}(\omega_1)=2\pi i.
 \end{equation}
 The following homogeneity property of $\zeta$ and $\eta$ will be very useful.
 \begin{prop}\label{prop2.1}
 	If $\La$ is a lattice and $\alpha\in\BC$ then
 	\begin{equation}\label{hom1}
\zeta(\alpha\La,\alpha z)\,=\,\alpha^{-1}\zeta(\La,z) \quad \mbox{and} \quad 	
\eta_{\alpha\La}(\alpha\omega)\,=\,\alpha^{-1}\eta_{\La}(\omega)\;.
 	\end{equation}
 \end{prop}
\begin{proof}
The first relation follows from the expansion \eqref{zeta} and the second relation follows from \eqref{quasi}
	\end{proof}
We refer to \eqref{hom1} by saying that $\zeta$ and $\eta$ are homogeneous of weight -1.

We now focus on lattices of the form $\La_{\tau}\;=\;\BZ+\tau\BZ$ where $\tau$ is in the upper half-plane $\fH=\{z\in \BC \mid \Im (z)>0\}$. From \eqref{expr1} we can readily see that the quasi-periods
$\eta_{\La_{\tau}}(1)$ and $\eta_{\La_{\tau}}(\tau)$ are meromorphic functions on $\fH$ and so is the function defined by
\begin{equation}\label{equiv1}
h(\tau)\,=\,\frac{\eta_{\La_{\tau}}(\tau)}{\eta_{\La_{\tau}}(1)}\,.
\end{equation}
For the remainder of this paper, if $\displaystyle \gamma=\begin{bmatrix}
	a&b\\c&d
\end{bmatrix}\in\GC$ and $z\in\BC$ we define the action $\gamma z$ by
\begin{equation}
\gamma z\,=\,\frac{az+b}{cz+d}\,.
\end{equation}
When $\gamma\in\SR$ and $z\in\fH$, then this is the usual action on $\fH$ by linear fractional transformation providing all the automorphisms of $\fH$.
\begin{prop}\cite{brady}
The function $h$ defined by \eqref{equiv1} satisfies
\begin{equation}\label{equiv}
	h(\gamma \tau)\,=\,\gamma h(\tau)\,  ,\  \gamma\in\SL\, ,\ \tau\in\fH.
\end{equation}
\end{prop}
\begin{proof}
	Let $\gamma\in\SL$ and $\tau\in\fH$. We have
	\begin{align*}
	\eta_{\La_{\g\tau}}(\g\tau)\,&=\,\eta_{\BZ+\frac{a\tau+b}{c\tau+d}\BZ}\left(\frac{a\tau+b}{c\tau+d}\right)\\
	&=(c\tau+d)\eta_{(c\tau+d)\BZ+(a\tau+b)\BZ}(a\tau+b)\  \ (\mbox{by homogeneity of }\eta)\\
	&=(c\tau+d)\eta_{\La_{\tau}}(a\tau+b)\\
	&=(c\tau+d)(a\eta_{\La_{\tau}}(\tau)+b\eta_{\La_{\tau}}(1))
	\end{align*}
where we have used the linearity of $\eta$ and the fact that the lattices
 $\La_{\tau}$ and $(c\tau+d)\BZ+(a\tau+b)\BZ$ are the same. Also,
 \begin{align*}
 \eta_{\La_{\g\tau}}(1)\,&=\,	(c\tau+d)\eta_{(c\tau+d)\BZ+(a\tau+b)\BZ}(c\tau+d)\  \ (\mbox{by homogeneity of }\eta)\\
 &=(c\tau+d)\eta_{\La_{\tau}}(c\tau+d)\\
 &=(c\tau+d)(c\eta_{\La_{\tau}}(\tau)+d\eta_{\La_{\tau}}(1)).
 	\end{align*}
 Therefore,
 \[
 h(\g\tau)=\frac{a\eta_{\La_{\tau}}(\tau)+b\eta_{\La_{\tau}}(1)}{c\eta_{\La_{\tau}}(\tau)+d\eta_{\La_{\tau}}(1)}=\frac{ah(\tau)+b}{ch(\tau)+d}=\g h(\tau)\,.
 \]
	\end{proof}

A meromorphic function on $\fH$ that satisfies \eqref{equiv} will be called equivariant with respect to $\SL$. We will expand more on these functions in later sections.

\section{Elliptic Zeta functions}\label{section3}
Following \cite{brady}, we will generalize the notion of Weierstrass zeta function and its quasi-periods. Let $\cl $ be the set of lattices $\La_{(\omega_1,\omega_2)}=\omega_1\BZ+\omega_2\BZ$ with $\Im(\omega_2/\omega_1>0)$. We define an elliptic zeta function of weight $k\in\BZ$ as a map
\[
\cz:\cl\times\BC\longrightarrow \BC\cup\{\infty\}
\]
satisfying the following properties:
\begin{enumerate}
	\item \label{c1} For each ${\La}=\omega_1\BZ+\omega_2\BZ$, the map
	\[
	\cz(\La,\cdot):\BC\longrightarrow \BC\cup \{\infty\}
	\]
	is  quasi-periodic, that is
	\[
	\cz(\La,z+\om)\,=\,\cz(\La,z)+H_{\La}(\om)\;,\ z\in\BC\;,\ \omega\in \La\;,
	\]
	where the quasi-period function $H_{\La}(\om)$ does not depend on $z$.
	\item \label{c2} $\cz$ is homogeneous of weight $k$ in the sense that
	\[
	\cz(\alpha \La, \alpha z)\,=\,\alpha^{k}\;\cz(\La,z)\;,\ \alpha\in\BC^*\;,\ z\in\BC.
	\]
	\item \label{c3} If $\La_{\tau}=\BZ+\tau\BZ$, $\tau\in\fH$, then the quasi-periods $H_{\La_{\tau}}(\tau)$ and $H_{\La_{\tau}}(1)$ as functions of $\tau$ are meromorphic on $\fH$.
\end{enumerate}

It follows from (\ref{c1}) that for each $\La$, the quasi-period function $H_{\La}$ is $\BZ-$linear, and therefore it is completely determined by $H_{\La}(\omega_1)$ and $H_{\La}(\omega_2)$. Moreover, we have the following result  generalizing \propref{prop2.1}.
\begin{prop}
	Let $\cz$ be an elliptic zeta function of weight $k$ et let $H_{\La}$ be the quasi-period function for each lattice $\La$. Then for all  $\alpha\in\BC^\times$ and $\omega\in \La$, we have
	\begin{equation}\label{hom2.3}
	H_{\alpha \La}(\alpha \omega)\,=\, \alpha^{k}\;H_{\La}(\omega).
	\end{equation}
\end{prop}
\begin{proof}
	On one hand we have:
	\begin{align*}
	\cz(\alpha \La,\alpha(z+\omega))\,&=\,\cz(\alpha \La, \alpha z)+H_{\alpha \La}(\alpha \omega)\\
	&=\,\alpha^{k}\cz(\La,z)+H_{\alpha \La}(\alpha \omega).
	\end{align*}
On the other hand, we have
\begin{align*}
	\cz(\alpha \La,\alpha(z+\omega))\,&=\,\alpha^{k}\cz(\La,z+\omega)\\
	&=\,\alpha^{k}(\cz(\La,z)+H_{\La}(\omega)),
	\end{align*}
and the proposition follows.
	\end{proof}
Notice that two elliptic zeta functions having the same quasi-period function must differ by an elliptic function. Simple examples are given by the identity map $z$, or the Weierstrass zeta function $\zeta(\La,z)$. We will see below that these two examples will, in a certain sense, generate all the other elliptic zeta functions.
Also, since for a fixed lattice the derivative of an elliptic zeta function with respect to $z$ is an elliptic function for the lattice, this provides a way to construct infinitely many of them by taking integrals of elliptic functions.

Let  $\omega_1$ and $\omega_2$ such that $\Im(\omega_2/\omega_1)>0$ and set
\[
M_{(\omega_1,\omega_2)}=\begin{bmatrix}
\omega_2 &\eta(\omega_2)\\
\omega_1& \eta(\omega_1)
\end{bmatrix},
\]
where $\eta$ is the quasi-period map of the Weierstrass zeta function
$\zeta(\omega_1\BZ+\omega_2\BZ,z)$
Using the Legendre relation \eqref{legendre}, we have
\[
 \det M_{(\omega_1,\omega_2)}=-2\pi i.
\]
Let $\cz$ be an elliptic zeta function of weight $k$ with the two quasi-periods $H(\omega_1)$ and $H(\omega_2)$. Set
\[
\begin{bmatrix}
\Phi\\ \Psi\,
\end{bmatrix}
=M_{(\omega_1,\omega_2)}^{-1}
\begin{bmatrix}
H(\omega_2)\\ H(\omega_1)
\end{bmatrix}.
\]
In other words
\begin{align}
	2\pi i\Phi&=\eta(\omega_2) H(\omega_1)- \eta(\omega_1) H(\omega_2)  \label{leg1}\\
	2\pi i\Psi&=\omega_1H(\omega_2)-\omega_2H(\omega_1) \label{leg2}.
	\end{align}
\begin{prop}\label{indep}
The quantities $\Phi$ and $\Psi$ do not depend on the choice of the basis $(\omega_1,\omega_2)$, and, as functions of the lattice $\omega_1\BZ+\omega_2\BZ$, they are homogeneous of respective weights $k-1$ and $k+1$.
\end{prop}
\begin{proof}
	Let $a$, $b$, $c$ and $d$ be integers such that $ad-bc=1$. The expressions in \eqref{leg1} and \eqref{leg2} are invariant if we change the basis $(\omega_1,\omega_2)$ to the basis $(a\omega_1+b\omega_2,c\omega_1+d\omega_2)$. Indeed, using the linearity of $\eta$ and $H$ we have for the expression of $\Phi$:
	\begin{align*}
\eta(c\omega_1&+d\omega_2)H(a\omega_1+b\omega_2)-\eta(a\omega_1+b\omega_2)H(c\omega_1+d\omega_2)\\
&=[c\eta(\omega_1)+d\eta(\omega_2)][aH(\omega_1)+bH(\omega_2)]-[a\eta(\omega_1)+b\eta_(\omega_2)][cH(\omega_1)+dH(\omega_2)]\\
&=H(\omega_1)\eta(\omega_2)-H(\omega_2)\eta(\omega_1)\;.
		\end{align*}
	Similar calculations hold for the expression of $\Psi$.
	 The values of the weights are straightforward knowing that the weight is $k$ for $H$, -1 for $\eta$ and 1 for both $\omega_1$ and $\omega_2$.
	\end{proof}
We can therefore denote $\Phi$ and $\Psi$ by $\Phi_{\La}$ and $\Psi_{\La}$ as they depend only on the lattice $\La$.

\begin{prop} \cite{brady}\label{can}
	Let $\cz$ be an elliptic zeta function of weight $k$ and quasi-period function $H$, and let $\Phi_{\La}$ and $\Psi_{\La}$ be as above.  Then for each lattice ${\La}$, there exists an elliptic function $E_{\La}$ such that
	\begin{equation}\label{can1}
	\cz({\La},z)\,=\,\Phi_{\La}\;z\;+\;\Psi_{\La}\zeta(z)\;+\;E_{\La}(z).
	\end{equation}
\end{prop}
\begin{proof}
	It is clear by construction of $\Phi$ and $\Psi$ that  the map $\Phi_{\La} z+\Psi_{\La}\zeta$ satisfies the conditions of a weight $k$ elliptic zeta function. Moreover, for each ${\La}=\omega_1\BZ+\omega_2\BZ$, the quasi-periods of
	$\Phi_{\La}\;z\;+\;\Psi_{\La}\;\zeta(z)$ are $\Phi_{\La}\;\omega_i\;+\;\Psi_{\La}\;\eta(\omega_i)$, $i=1,\,2$, which coincide with the quasi-periods $H(\omega_1)$ and $H(\omega_2)$ of $\cz$ as we have
	$\displaystyle \begin{bmatrix}  H(\omega_2)\\H(\omega_1)\end{bmatrix}=
	 \begin{bmatrix}\omega_2&\eta(\omega_2)\\ \omega_1&\eta(\omega_1) \end{bmatrix} \begin{bmatrix} \Phi_{\La}\\\Psi_{\La} \end{bmatrix}$, and therefore the two elliptic zeta functions differ by an elliptic function for the lattice ${\La}$.
\end{proof}
It is clear that the expression \eqref{can1} for an elliptic zeta function is unique up to the elliptic function $E_{\La}(z)$ since $\Phi_{\La}$ and $\Psi_{\La}$ are uniquely determined. Moreover, we view the relations \eqref{leg1} and \eqref{leg2} as the generalization for an elliptic zeta function of the Legendre relation \eqref{legendre} for the Weierstrass zeta function. Finally, using a similar proof to that of \propref{equiv}, we have
\begin{prop}
	Let $\cz$ be an elliptic zeta function with a quasi-period map $H$. For $\tau\in\fH$ and ${\La}_{\tau}=\BZ+\tau\BZ$, suppose that $H(1)$ is not identically zero, then the meromorphic function $h(\tau)=H(\tau)/H(1)$ is equivariant with respect to $\SL$.
\end{prop}

\section{Modular forms}\label{section4}

In this section we will investigate the connection between
elliptic zeta functions and modular forms for $\SL$.
In the following theorem, we will show that each elliptic zeta functions gives rise
to a weight 2 (meromorphic) modular form for $\SL$, and conversely, each weight 2 modular form yields an elliptic zeta function.

\begin{thm}\label{thm4.1}
Let $\cz$ be an elliptic zeta function with $\Phi_{\La}$ and $\Psi_{\La}$ as in \eqref{can1} and suppose $\Psi_{{\La}_{\tau}}$ is not identically zero as a function of $\tau$. Then the map
\begin{equation}\label{map1}
\cz \longmapsto \frac{\Phi_{{\La}_{\tau}}}{\Psi_{{\La}_{\tau}}}
\end{equation}
is well defined between the set of elliptic zeta functions and the space of weight 2 modular forms $M_2(\SL)$. In addition, this map is surjective.
\end{thm}
\begin{proof}
	Let $k\in\BZ$ be the weight of $\cz$ and set
\[
f(\tau)\,=\,\frac{\Phi_{{\La}_{\tau}}}{\Psi_{{\La}_{\tau}}}\,,\ \tau\in\fH.
\]
Since $\Phi_{{\La}_{\tau}}$ and $\Psi_{{\La}_{\tau}}$ are meromorphic in $\tau$, so is $f(\tau)$. Now let $\displaystyle \gamma=\begin{bmatrix}
a&b\\c&d
\end{bmatrix}\in\SL$. Since $\Phi_{\La}$ and $\Psi_{\La}$ are homogeneous of weights $k-1$ and $k+1$ respectively,  we have
\[
\Phi_{{\La}_{\gamma\tau}}\;=\;(c\tau +d)^{-k+1}\Phi_{(a\tau+b)\BZ+(c\tau+d)\BZ} \;=\;
(c\tau +d)^{-k+1}\Phi_{{\La}_{\tau}}
\]
and
\[
\Psi_{{\La}_{\gamma\tau}}\;=\;(c\tau +d)^{-k-1}\Psi_{(a\tau+b)\BZ+(c\tau+d)\BZ} \;=\;
(c\tau +d)^{-k-1}\Psi_{{\La}_{\tau}}.
\]
Therefore
\[
f(\gamma\tau)\,=\,(c\tau+d)^2\,f(\tau)\;.
\]
Hence the map is well defined as $\Phi_{\La}$ and $\Psi_{\La}$ are uniquely determined by $\cz$.
We now prove that the map is onto. Let $f\in M_2(\SL)$ and set, for ${\La}=\omega_1\BZ+\omega_2\BZ$,
\[
\Phi_{\La}\,=\, \frac{1}{\omega_1^2}\;f\left(\frac{\omega_2}{\omega_1}\right)\,,\ \Psi_{\La}\,=\,1\,.
\]
The map $\Phi_{\La}$ is well defined in the sense that it is independent of the choice of the basis $(\omega_1,\omega_2)$. Indeed, if $\displaystyle \gamma=\begin{bmatrix}
a&b\\c&d
\end{bmatrix}\in\SL$, then
\begin{align*}
\frac{1}{(a\omega_1+b\omega_2)^2}\;f\left(\frac{c\omega_1+d\omega_2}{a\omega_1+b\omega_2}\right)\;&=\;\frac{1}{(a\omega_1+b\omega_2)^2}\;f\left(\frac{d\;\frac{\omega_2}{\omega_1}+c}{b\;\frac{\omega_2}{\omega_1}+a}\right)\\
&=\; \frac{(b\;\frac{\omega_2}{\omega_1}+a)^2}{(a\omega_1+b\omega_2)^2}\;f\left(\frac{\omega_2}{\omega_1}\right)\\
&=\;\frac{1}{\omega_1^2}\;f\left(\frac{\omega_2}{\omega_1}\right)\,.
\end{align*}
Thus, we have an elliptic zeta function
\[
\cz({\La},z)\;=\; \frac{1}{\omega_1^2}\;f\left(\frac{\omega_2}{\omega_1}\right)\,z+\zeta(z)
\]
of weight -1 that is sent to $f(\tau)$ by the map \eqref{map1}.
	\end{proof}

\section{Equivariant functions}\label{section5}
We introduced the notion of equivariant functions earlier as being meromorphic functions on $\fH$ that commute with the action of the modular group. They were
extensively studied in \cite{sb1,sb2,ss2,ss1} in connection with modular forms, vector-valued modular forms and other topics. In particular, each modular form of any weight (even with a character) gives rise to an equivariant function. Indeed,
if $f$ is a modular form of weight $k$, then the function
\[
h_f(\tau)\;=\;\tau+k\;\frac{f(\tau)}{f'(\tau)}
\]
is equivariant with respect to $\SL$.

 Not all the equivariant functions arise in this way from a modular form. In fact a necessary and sufficient condition for an equivariant function $h$ to be equal to $h_f$ for some modular form $f$ is that the poles  of $\displaystyle 1/(h(\tau)-\tau)$ in $\fH\cup\{\infty\}$ are all simple with  rational residues \cite{sb2}. Such functions are referred to as the rational equivariant functions.

Important applications were obtained regarding the critical points of modular forms and
their $q-$expansion \cite{ss3}. As an example, recall the Eisenstein series $G_2(\tau)$ defined by
\[
G_2(\tau)=\frac{1}{2}\sum_{n\neq0}\,\frac{1}{n^2}\,+\,\frac{1}{2}\sum_{m\neq 0}\sum_{n\in\BZ}\,\frac{1}{(m\tau+n)^2}
\]
and the normalized weight two Eisenstein series
\[
E_2(\tau)=\frac{6}{\pi^2}G_2(\tau)=1-24\,\sum_{n=1}^{\infty}\,\sigma_1(n)q^n\,,\ \ q=e^{2\pi i\tau},
\]
where $\sigma_1(n)$ is the sum of positive divisors of $n$.
On can easily deduce from the definition of the Weierstrass $\zeta-$function that \cite{lang}
\begin{align*}
	\eta(1)&=G_2(\tau)\\
	\eta(\tau)&=\tau G_2(\tau)-2\pi i,
	\end{align*}
and since we have
\[
E_2(\tau)=\frac{1}{2\pi i} \frac{\Delta'(\tau)}{\Delta(\tau)},
\]
where $\Delta$ is the weight 12 cusp form (the discriminant)
\[
\Delta(\tau)=q\,\prod_{n=1}^{\infty}\,(1-q^n)^{24}\,,\ \ q=e^{2\pi i\tau},
\]
we deduce
\begin{prop}
	The equivariant function $\displaystyle h(\tau)=\frac{\eta(\tau)}{\eta(1)}$ from \propref{equiv} is rational with
	\[
	h(\tau)\,=\,\tau+12\;\frac{\Delta}{\Delta'}\,.
	\]
\end{prop}
 Let's denote by $Eq$ the set of all equivariant functions with respect to $\SL$.
Although $h(\tau)=\tau$ is trivially equivariant, it will be excluded from $Eq$.

Recall that if  $f\in M_2(\SL)$ and $\displaystyle \gamma=\begin{bmatrix}a&b\\c&d\end{bmatrix}\,\in \GL_2(\BC)$ we denote  by $\gamma f(\tau)$
\[
\gamma f(\tau)\,=\,\frac{af(\tau)+b}{cf(\tau)+d}\;.
\]
Now recall from \S3 the matrix
\[
M_{(1,\tau)}\;=\;\begin{bmatrix}
\tau & \eta(\tau)\\1&\eta(1)
\end{bmatrix}
\]
which is invertible thanks to the Legendre relation.
\begin{thm}\label{bij1}
	The map from $M_2(\SL)$ to $Eq$
	\[
	f\longmapsto M_{(1,\tau)}\,f
	\]
	is a bijection. The inverse map is given by
	\[
	h\longmapsto M_{(1,\tau)}^{-1}\,h\;.
	\]
\end{thm}
\begin{proof}
	Let $f\in M_2(\SL)$ and set
	\[
	h(\tau)\;=\; M_{(1,\tau)}\,f(\tau)\;=\;\frac{\tau f(\tau)+\eta(\tau)}{f(\tau)+\eta(1)}\;
	\]
For $\displaystyle \gamma=\begin{bmatrix}a&b\\c&d\end{bmatrix}\,\in \SL$, we have
\begin{align*}
	h(\gamma\tau)&\;=\; \frac{\gamma \tau f(\gamma\tau)+\eta_{\Lambda_{\gamma\tau}}(\gamma\tau)}{f(\gamma\tau)+\eta_{\Lambda_{\tau}}(1)}\;.    \\
	\end{align*}
Since
\[
\eta_{\Lambda_{\gamma\tau}}(\gamma\tau)=(c\tau+d)(a\eta(\tau)+b\eta(1)),
\]
\[
\eta_{\Lambda_{\tau}}(1)=(c\tau+d)(c\eta(\tau)+d\eta(1))
\]
and
\[
f(\g \tau)=(c\tau +d)^2f(\tau)\;,
\]
we have
\[
h(\gamma\tau)=\frac{(a\tau+b)f(\tau)+a\eta(\tau)+b\eta(1)}{(c\tau+d)f(\tau)+c\eta(\tau)+d\eta(1)}=\gamma h(\tau)\;.
\]
Similarly, one can prove that if $h\in Eq$, then  $M_{(1,\tau)}^{-1}\,h\in M_2(\SL)$.
	\end{proof}

	Usually the definition of a meromorphic modular forms involves also the behavior at the cusps. More precisely if $f$ is a modular form for $\SL$, then $f(\tau+1)=f(\tau)$ for all $\tau\in\fH$ and thus has a Fourier expansion which is a Laurent series in $q=\exp(2\pi i\tau)$. We say that $f$ is meromorphic at the cusp $\infty$ if this Laurent series has only finitely many negative powers of $q$. In the meantime, if $h$ is equivariant for $\SL$ then $h(\tau+1)=h(\tau)+1$. Hence $h(\tau)-\tau$ is also periodic of period one and thus has a Fourier expansion in $q$. The proper behavior of $h$ at the cusp at infinity is that
	$h(\tau)-\tau$ is meromorphic in $q$, see \cite{sb2}.  If a weight two modular form $f$ and an equivariant function $h$ correspond to each other by \thmref{bij1}, then $\displaystyle h(\tau)=\frac{\tau f(\tau)+\eta(\tau)}{f(\tau)+\eta(1)}$ and thus, using the Legendre relation \eqref{legendre}, we have
\[
h(\tau)-\tau\,=\,\frac{2\pi i}{f(\tau)+\eta(1)}.
\]
Since $\eta(1)=G_2(\tau)$ is holomorphic in $q$ we see that the behavior at infinity for both $f$ and $h$ is preserved under the correspondence of \thmref{bij1}.

Taking into account the results of the above sections, we have thus established a correspondence between the set of elliptic zeta functions,
the space of modular forms of weight 2 for $\SL$ and the set of equivariant functions for $\SL$ summarized as follows

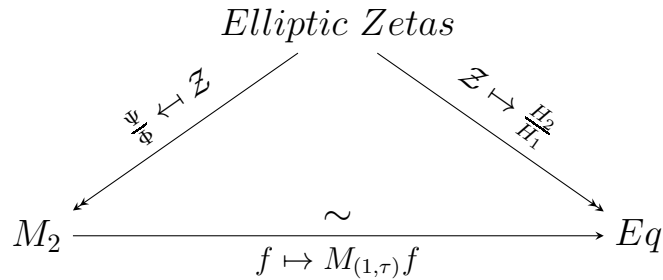
\begin{center}
\begin{tikzpicture}[scale=2]
\node (Z) at (0,1.4) {\large $Elliptic\  Zetas$};
\node (E) at (2,0) {\large $Eq$}; 
\node (M) at (-2,0) {\large $M_2$};  
\node [above] at (0,0) {\large$ \sim $} ;
\draw [>=stealth,->>] (Z) to (E);
\draw [>=stealth,->] (M) to (E);
\draw [>=stealth,->>] (Z) to (M);
\node  [above,rotate=35] at (-1,0.7) {$\frac{\Psi}{\Phi}\mapsfrom\mathcal Z$};
\node  [above,rotate=-35] at (1,0.7) { $\mathcal Z\mapsto \frac{H_2}{H_1}$};
\node  [below,rotate=0] at (0,0) { $f\mapsto M_{(1,\tau)} f$};
\end{tikzpicture}
\end{center}

\noindent where $H_1$ and $H_2$ are the quasi-periods of the elliptic zeta function $\cz$, 
$\Phi$ and $\Psi$ are such that $\cz(\Lambda,z)=\Phi\;z+\Psi\zeta(z)+E$ with $E$ elliptic and $M_{(1,\tau)}$ as above. Of course, this diagram is commutative and each map is surjective.

\section{The case of modular subgroups}\label{section6}

So far the constructions in the previous sections involve the full modular group $\SL$. In the meantime, the notion of modular forms or equivariant functions can be restricted to any finite index subgroup. Thus we need to define the notion of
elliptic zeta functions for any such subgroup.

Fix a modular subgroup $\Gamma$ of finite index in $\SL$. Set
\[
{\mathbb M}=\{(\omega_1,\omega_2)\in\BC^2:\, \Im (\omega_2/\omega_1)>0\}.
\]
The group $\Gamma$ acts on $\mathbb M$ in the usual way:
\[
\gamma (\omega_1,\omega_2)=(\omega_1,\omega_2)\gamma^t.
\]
Denote by $\Omega_\Gamma$ the quotient $\Gamma\backslash \mathbb M$ and the class of $(\omega_1,\omega_2)$ by $[\omega_1,\omega_2]$. Also, $\BC^*$ acts on $\mathbb M$ in the usual way and this action extends to $\Omega_\Gamma$ as:
\[
\alpha[\omega_1,\omega_2]=[\alpha\omega_1,\alpha\omega_2].
\]
If $\Gamma=\SL$, then $[\omega_1,\omega_2]$ is identified with the lattice
$\Lambda_{(\omega_1,\omega_2)}=\omega_1\BZ+\omega_2\BZ$, but for an arbitrary finite index subgroup $\Gamma$, the situation is different. Following the ideas in \cite{con}, $\Omega_\Lambda$ is identified with the set of pairs of lattices $(\Lambda,\Lambda')$ with $\Lambda'$ being a finite index sub-lattice of $\Lambda$ fixed by $\Gamma$ and $\Lambda'$ is the smallest such lattice (and thus defined as the intersection of all such sub-lattices that are $\Gamma-$invariant).
If such pair $(\Lambda, \Lambda')$ is given, and as $\SL$ acts by automorphisms of $\Lambda$ by a change of basis, $\Gamma$ would be defined by
\[
\Gamma=\{\gamma\in\SL:\, \gamma\Lambda'\subseteq \Lambda' \}.
\]
For example, if $\Gamma=\Gamma(N)$ is the principal congruence subgroup of level $N\geq 1$, then $\Lambda'=N\omega_1\BZ+N\omega_2\BZ$ which is a sub-lattice of $\omega_1\BZ+\omega_2\BZ$ of index $N^2$ . If $\Gamma=\Gamma_0(N)$, then $\Lambda'=\omega_1\BZ+N\omega_2\BZ$ of index $N$ in
 $\omega_1\BZ+\omega_2\BZ$. However, we will not need this identification in what follows.

 A $\Gamma-$elliptic zeta function with respect to $\Gamma$ is a map
 \[
 \cz:\Omega_\Gamma \times \BC\longrightarrow \BC\cup\{\infty\}
 \]
 satisfying
 \begin{enumerate}
 	\item For each $[\omega_1,\omega_2]\in\Omega_\Gamma$, the map
 		\[
 		\cz([\omega_1,\omega_2],\cdot):\BC\longrightarrow \BC\cup\{\infty\}
 		\]
 		is quasi-periodic with respect to $\Lambda_{(\omega_1,\omega_2)}$, that is,
 		for all $z\in\BC$ and all $\omega\in \Lambda_{(\omega_1,\omega_2)}$ we have
 			\[
 				\cz([\omega_1,\omega_2],z+\omega)=
 						\cz([\omega_1,\omega_2],z)+H_{[\omega_1,\omega_2]}(\omega).
 			\]
 			\item The map $\cz$ is homogeneous, that is, there exists an integer $k$, referred to as the weight of $\cz$, such that for all $\alpha\in\BC^*$,
 			$[\omega_1,\omega_2]\in\Omega_\Gamma$ and $z\in\BC$ we have
 			\[
 			\cz(\alpha[\omega_1,\omega_2], \alpha z)=\alpha^k
 				\cz([\omega_1,\omega_2],z).
 			\]
 			\item The maps
 			\[
 			\tau\mapsto H_{[1,\tau]}(\tau)\,\mbox{ and }\,\tau\mapsto
 				H_{[1,\tau]}(1)
 			\]
 			are meromorphic in $\fH$.
 \end{enumerate}
From this definition, it is clear that the quasi-period map $H_{[\omega_1,\omega_2]}$ is $\BZ-$linear on the lattice $\Lambda_{(\omega_1,\omega_2)}$
and thus it is completely determined by its values on $\omega_1$ and $\omega_2$. It is also homogeneous of weight $k$:
\[
H_{[\alpha\omega_1,\alpha\omega_2]}(\alpha\omega)=\alpha^k
H_{[\omega_1,\omega_2]}(\omega)\,,\ \omega\in\Lambda_{(\omega_1,\omega_2)}\,,\ \alpha\in\BC^*.
\]
Using the same arguments as in \secref{section3}, one can easily establish the following
\begin{prop}
	Let $\cz:\Omega_\Gamma\times \BC \longrightarrow \BC\cup \{\infty\}$ be a
$\Gamma-$elliptic zeta function. There exists unique maps $\Phi_{[\omega_1,\omega_2]}$ of weight $k-1$ and $\Psi_{[\omega_1,\omega_2]}$
of weight $k+1$ such that for all $[\omega_1,\omega_2]\in\Omega_\Gamma$ and $z\in\BC$ we have:
\[
\cz([\omega_1,\omega_2],z)=\Phi_{[\omega_1,\omega_2]}\,z + \Psi_{[\omega_1,\omega_2]}\,\zeta(z)+E_{[\omega_1,\omega_2]}(z)\;,
\]
where $E_{[\omega_1,\omega_2]}(z) $ is an $\La_{(\omega_1,\omega_2)}-$elliptic function.
\end{prop}
Notice that $\Phi_{[\omega_1,\omega_2]}$ and $\Psi_{[\omega_1,\omega_2]}$ can be shown to be independent of the choice of the representative of the class $[\omega_1,\omega_2]$ in the same way as for \propref{indep} using transformations from $\Gamma$ instead of $\SL$.

 Let $M_2(\Gamma)$ denote the space of meromorphic weight two modular forms with
respect to $\Gamma$ and $Eq(\Gamma)$ be the set of $\Gamma-$equivariant functions, that is the set of meromorphic functions on $\fH$ which commute with the action of $\Gamma$. It is clear that the matrix $M_{(1,\tau)}$ of the previous section provides a bijection between $M_2(\Gamma)$  and $Eq(\Gamma)$. Using the fact that by definition of $\Omega_\Gamma$, when
$\displaystyle \gamma=\begin{bmatrix} a&b\\c&d\end{bmatrix}\in\Gamma$, then
$[\omega_1,\omega_2]=[a\omega_1+b\omega_2,c\omega_1+d\omega_2]$, we deduce, in the same way as in the previous sections, the following

\begin{thm}
	If $\Gamma$ is a finite index subgroup of $\SL$, then
	\begin{enumerate}
		\item The map
		\begin{equation} \label{map111}
		\cz\mapsto  \frac{H_{[1,\tau]}(\tau)}{H_{[1,\tau]}(1)}
		\end{equation}
		is a well defined map from the set of $\Gamma-$elliptic zeta functions to $Eq(\Gamma)$.
		\item
		The map
		\begin{equation}
		\label{map112}
		\cz\mapsto \frac{\Phi_{[1,\tau]}}{\Psi_{[1,\tau]}}
		\end{equation}
		is well defined  between the set of $\Gamma-$elliptic zeta functions and $M_2(\Gamma)$. It is also onto as for each $f\in M_2(\Gamma)$
	\[
	\cz([\omega_1,\omega_2],z)\,=\,\frac{1}{{\omega_1}^2}\;f\left(\frac{\omega_2}{\omega_1}\right)\;z+\zeta(z)
	\]
	is a $\Gamma-$elliptic zeta function of weight -1 that maps to $f$ by \eqref{map112}.
	\end{enumerate}
\end{thm}

\begin{remark}\rm
	Using the bijection between $M_2(\G)$ and $Eq(\Gamma)$ and the surjective map \eqref{map112}, one can also see that the map  \eqref{map111} is surjective. Thus we have shown that each $\Gamma-$equivariant function arises from a $\Gamma-$elliptic zeta function. Notice that the trivial equivariant function $\tau$ is also the quotient of the quasi-periods of the trivial $\Gamma-$elliptic zeta function $\cz(z)=z$.
\end{remark}
\begin{remark}\rm It is worth explaining the behavior at the cusps as has been discussed at the end of \secref{section5} for the cusp at infinity. In the case of a modular subgroup $\G$ of $\SL$, there are more than one cusp which are not in the same $\Gamma-$orbit. Meanwhile, the analytic behavior of a meromorphic modular form at a rational cusp  is well defined, see Chapter 1 of  \cite{shimura} for instance, and that of an equivariant function has been established in \cite{sb2}, \S 3. It is not difficult to show that the two behaviors at a rational cusp are well preserved under the correspondence between a weight 2 modular form for $\Gamma$ and a $\G-$equivariant function when $\Gamma$ is a finite index subgroup of $\SL$.
	
\end{remark}

 \section{Examples} \label{section7}

In this section, we study an important class of elliptic functions given by integrals of the powers $\wp^n$ of the Weierstrass $\wp-$function. These integrals were treated in \cite{ss1}.

Let $\La=\omega_1\BZ+\omega_2\BZ $, $\Im(\omega_2/\omega_1)>0$, be a lattice in $\BC$. The Eisenstein series $g_2$ and $g_3$ are defined by
\[
g_2(\La)\,=\,60\;\sum_{\omega\in\La-\{0\}}\,\frac{1}{\omega^4}\;,
\]
\[
g_3(\La)\,=\,140\;\sum_{\omega\in\La-\{0\}}\,\frac{1}{\omega^4}\;.
\]
When $\Lambda=\BZ+\tau\BZ$, $\tau\in\fH$, $g_2$ and $g_3$, as functions of $\tau$ are modular forms of weight four and six respectively.
For a non-negative integer $n$, the power $\wp^n(z)$ can be written as a linear combination of 1, $\wp$ and successive derivatives of $\wp$:
\begin{equation}\label{pn}
\wp^n(\Lambda,z)\;=\;\Phi_n(\Lambda)-\Psi_n(\Lambda)\wp(\Lambda,z)+\sum_{k=1}^{n-1}\,\alpha_k\wp^{(2k)},
\end{equation}
where  the coefficients $\alpha_k$ are polynomials in $g_2$ and $g_3$
with rational coefficients, see \cite{tan-molk}, page 108. In particular, $\Phi_0=1$, $\Psi_0=0$, $\Phi_1=0$ and $\Psi_1=-1$.

For each lattice $\Lambda$ and $z\in\BC$, a primitive $\displaystyle \int \wp^n(u) du$ of $\wp^n$ has the form
\begin{equation}
\Phi_n(\Lambda)\;z + \Psi_n(\Lambda)\zeta(\Lambda,z)+E_n(\Lambda, z),
\end{equation}
where for each $\Lambda$,  $E_n(\Lambda,z)$ is a $\Lambda-$elliptic function. We define
\[
\cz_n(\Lambda,z)\;:=\;\Phi_n(\Lambda)\;z + \Psi_n(\Lambda)\zeta(\Lambda,z).
\]
It is clear that for each $\Lambda$, $\cz_n(\Lambda,z)$ is quasi-periodic with the quasi-period map given by
\[
H_n(\omega)=\Phi_n(\Lambda) \;\omega +\Psi_n(\Lambda)\eta(\omega),
\]
where $\eta$ is the quasi-period map for the Weierstrass $\zeta-$function.
If there is no confusion, we will write $\Phi_n$ for $\Phi_n(\Lambda)$ and $\Psi_n$ for $\Psi_n(\Lambda)$. According to \cite{tan-molk}, page 109 (see also \cite{ss1}, \S9), $\Phi_n$, $\Psi_n$, and thus $H_n$ satisfy the same three-term recurrence relation
\[
u_{n+1}=\frac{2n-1}{4(2n+1)}\;g_2u_{n-1} +\frac{n-1}{2(2n+1)}\;g_3u_{n-2},
\]
with the following initial conditions
\begin{align*}
	&\Phi_{-1}=\Psi_{-1}=H_{-1}=0\\
	&\Phi_0=1\,,\ \Psi_0=0\,,\ H_0(\omega)=\omega\\
	&\Phi_1=0\,,\ \Psi_1=-1\,,\ H_1(\omega)=-\eta(\omega).
	\end{align*}
One can easily see that when $\Lambda=\Lambda_{\tau}=\BZ+\tau\BZ$, $\tau\in\fH$,  $H_n$ is polynomial in $g_2$, $g_3$, $\eta(1)$ and $\eta(\tau)$, and thus $H_n(1)$ and $H_n(\tau)$ are meromorphic functions of $\tau$. It follows that the map $\cz_n(\Lambda,z)$ satisfies the axioms of a Weierstrass elliptic zeta function of weight $-2n+1$.

Let us put $\Phi_n(\tau)=\Phi_n(\Lambda_{\tau})$, $\Psi_n(\tau)=\Psi_n(\Lambda_{\tau})$, $\eta_1=\eta(1)$ and $\eta_2=\eta(\tau)$. Then $\displaystyle \Phi_2(\tau)=\frac{1}{12}\;g_2(\tau)$, $\displaystyle \Phi_3(\tau)=\frac{1}{10}\;g_3(\tau)$, $\displaystyle \Psi_2(\tau)=0$ and
$\displaystyle \Psi_3(\tau)=\frac{-3}{20}\;g_2(\tau)$. More generally, one can show by induction that
\begin{prop}
For each positive integer $n$, $\Phi_n$ and $\Psi_n$ are weighted homogeneous polynomials in $g_2$ and $g_3$ with rational coefficients and of degrees $2n$ and $2(n-1)$ respectively, and these degrees are also their weights as holomorphic modular forms.
\end{prop}
For small weights, it is clear that $\Phi_n$ and $\Psi_n$ are simple monomials.
In light of \S4 and \S5, for each elliptic zeta function $\cz_n$, there correspond, on one hand, a weight two modular form
\[
 f_n(\tau)=\frac{\Phi_n(\tau)}{\Psi_n(\tau)},
\]
which is a rational function of $g_2$ and $g_3$ with rational coefficients, and on the other hand, an equivariant function
\[
h_n(\tau)=\frac{H_n(\tau)}{H_n(1)}=\frac{\Phi_n(\tau)\;\tau+\Psi_n(\tau)\eta_2}{\Phi_n(\tau)+\Psi_n(\tau)\eta_1}.
\]
Also, using the Legendre relation, $f_n$ and $h_n$ are related by
\[
h_n(\tau)\;=\;\tau+\frac{2\pi i}{f_n(\tau) +\eta_1}.
\]
The following table gives $\Phi_n$, $\Psi_n$, $f_n$ and $h_n$ for $1\leq n\leq6$

\begin{center}
$$\displaystyle	\begin{array}{ |c|c|c|c|c| }
		\hline
		&&&&\\
	n & \Phi_n &\Psi_n & f_n & h_n\\
		&&&&\\
	\hline
		&&&&\\
	1&0&-1&0&\frac{\eta_2}{\eta_1}=\tau+\frac{2\pi i}{\eta_1}\\
		&&&&\\
	2&\frac{1}{12}\;g_2&0&-&\tau\\
		&&&&\\
	3&\frac{1}{10}\;g_3&\frac{-3}{20}\;g_2&\frac{-2}{3}\;\frac{g_3}{g_2}&\tau+\frac{6\pi ig_2}{-2g_3+3g_2\eta_1}\\
		&&&&\\
	4&\frac{5}{336}\;g_2^2&\frac{-2}{14}\;g_3&\frac{-5}{48}\;\frac{g_2^2}{g_3}&
	\tau+\frac{96\pi i g_3}{-5g_2^2+48g_3\eta_1}\\
		&&&&\\
	5&\frac{1}{30}\;g_2g_3&\frac{-7}{240}\;g_2^2&\frac{-8}{7}\;\frac{g_3}{g_2}&
	\tau+\frac{14\pi ig_2}{-8g_3+7g_2\eta_1}\\
		&&&&\\
	6&\frac{15}{4928}\;g_2^3+\frac{1}{55}\;g_3^2&\frac{-87}{1540}\;g_2g_3&\frac{-25}{464}\;\frac{g_2^2}{g_3}-\frac{28}{87}\;\frac{g_3}{g_2}&\tau+\frac{2784\pi ig_2g_3}{-75g_2^3-448g_3^2+1392g_2g_3\eta_1}\\
		&&&&\\
		\hline
	\end{array}
$$
\end{center}


\end{document}